\documentclass[]{amsart}

\usepackage{amssymb}
\usepackage{amsbsy}
\usepackage{amscd}
\usepackage{amsmath}
\usepackage{amsthm}
\usepackage{graphicx}
\numberwithin{defn}{section}
\newtheorem{thm}{Theorem}\numberwithin{thm}{section}
\numberwithin{lem}{section}
\numberwithin{prop}{section}
\newtheorem{remark}{Remark}\numberwithin{remark}{section}
\newtheorem{example}{Example}\numberwithin{example}{section}

\title[Orthogonal nets and full isothermic tori]
{A comment on full discretized 
isothermic tori in Euclidean spaces}

\author[]{K. Leschke}
\address{School of Computing and Mathematical Sciences\\ University of Leicester\\
  Leicester\\UK}
\email{k.leschke@leicester.ac.uk}
\author[]{F. Pedit}
\address{Department of Mathematics and Statistics\\ University of Massachusetts Amherst\\
  Amherst, MA 01030\\USA}
\email{pedit@math.umass.edu}
\author[]{W. Rossman}
\address{Faculty of Science\\ Kobe University\\
  Kobe 657-8501\\Japan}
\email{wayne@math.kobe-u.ac.jp}



\begin{document}

\begin{abstract}
Using discretized orthogonal systems (curvature line 
systems) with periodicity created using Darboux transformations and their 
permutability, we have discrete
and semi-discrete $k$-dimensional isothermic tori which are full in
$n$-dimensional Euclidean space $\mathbb{R}^n$, for any natural 
numbers $2 \leq k \leq n$.  
\end{abstract}

\maketitle

\section{Introduction}\label{sec1}

Darboux transformations of curves in $n$-dimensional Euclidean space $\mathbb{R}^n$ and their 
permutability properties provide for orthogonal nets, which, as we note here, readily 
offer a way to construct discrete
and semi-discrete $k$-dimensional isothermic tori which are {\it full} in 
$\mathbb{R}^n$, that is, do not lie in any $(n-1)$-dimensional affine subspace.
A possible corresponding statement about fullness of entirely smooth isothermic 
tori is not so immediately obtainable from the transformation theory of curves, as 
transforms are a discrete process, 
prompting us to comment about the discretized cases here:

\begin{thm}\label{mainthm}
For any natural numbers $2 \leq k \leq n$, there exist full isothermic discrete and semi-discrete 
$k$-dimensional tori in $n$-dimensional Euclidean space.  
\end{thm}

The tools we use to show this are well known, as described in 
the next section, and then we give the proof in the section after.  

The desired $k$-dimensional tori can be produced in two steps:
\begin{enumerate}
\item Choose initial conditions and spectral parameters for an initial set of Darboux transformations 
of a given polarized loop so they are also loops and, together with the first given loop, are 
full in $\mathbb{R}^n$.  
\item Take further transforms via permutability 
to create a Bianchi cube of transforms that contain full toroidal $k$-dimensional subnets.  
\end{enumerate}
In more detail, beginning with the initial planar loop (we will use a circle), either smooth or discrete, one can then 
include $n-2$ distinct Darboux transformations of that loop which are also loops.  
With suitable choices of initial conditions for the transforms one can guarantee that these $n-1$ 
loops form a full set in $\mathbb{R}^n$.  Using permutability to obtain all subsequent Darboux 
transformations (a useful formulation of this process can be found in \cite{B}), 
which are automatically loops, one has a Bianchi cube from which to 
choose a $k$-dimensional toroidal subnet.  It is possible to choose that subnet so that the initial circular 
loop and the first layer of $n-2$ Darboux transformations are all included, making the subnet full in 
$\mathbb{R}^n$.  This subnet is then discrete in $k-1$ 
directions, and either smooth or discrete in the single direction of the initial loop, so it is, 
respectively, a semi-discrete or completely discrete full isothermic $k$-torus, proving Theorem 
\ref{mainthm}.  Our goal is to explain this in detail.  

\section{Preliminaries}

\subsection{M\"obius geometry for curves}

We begin with the light cone model for M\"obius geometry, needed for discussing some tools 
we will describe.  Some suitable references on M\"obius geometry are \cite{Bl}, \cite{BS} and \cite{H}, and 
references for families of flat connections of curves and surfaces are \cite{BHMR}, \cite{Bnext2}, \cite{BHRS} and 
\cite{BHRS2}, with \cite{BHMR} in particular dealing with curves.  

Take ($n+2$)-dimensional Lorentz space $\mathbb{R}^{n+1,1}$ with metric 
$\langle \cdot , \cdot \rangle$ of signature $(+,...,+,-)$ and with light cone
\[ L^{n+1} = \{ X \in \mathbb{R}^{n+1,1} \, | \,  \langle X,X \rangle = 0 \} \; . \]
Consider a smooth curve into the light cone, obtained by lifting a given smooth curve $x(t)$ 
in $\mathbb{R}^n$, 
\[ X=X(t)= \begin{pmatrix} 2 x \\ 1-|x|^2_{\mathbb{R}^n} \\ 1+|x|^2_{\mathbb{R}^n} \end{pmatrix} , \;\;\; x=x(t) \in \mathbb{R}^n \; . \] 
Note that $X(t)$ lies in the set 
\[ \{ X \in L^{n+1} \, | \, \langle X, q \rangle =-1 \} \; , \;\;\; q=\begin{pmatrix} \vec{0} \\ -1/2 \\ 1/2 \end{pmatrix} \; , \] 
which is isometric to $\mathbb{R}^n$ when given the 
metric induced by $\mathbb{R}^{n+1,1}$.  

\begin{remark}
This set is most naturally considered as the conformal $n$-sphere of 
M\"obius geometry after projectivizing $L^{n+1}$ to $PL^{n+1}=L^{n+1}/\sim$, where 
$a \sim b$ if there exists a nonzero real $\alpha$ such that $a = \alpha b$.  This 
allows for inclusion of other $n$-dimensional spaces that are conformally equivalent to $\mathbb{R}^n$ 
and for which Darboux transforms are no different.  However, since we are restricting 
to Euclidean space, we just use the set as described above.  
\end{remark}

We give this curve $X(t)$, equivalently $x(t)$, a polarization $\frac{dt^2}{m}$ on its domain, where $m=m(t)$ is a nowhere vanishing 
smooth function.  The associated family of flat $SO_{n+1,1}$-connections is 
\[ \nabla^\lambda=d+\lambda 
\frac{X \wedge X_t}{m \cdot \langle X_t , X_t \rangle} dt \; , \]
where $\lambda \in \mathbb{R}$ is called the spectral parameter, 
\[ X_t := dX/dt = 
\begin{pmatrix}
2 x_t \\ -2 \langle x,x_t \rangle_{\mathbb{R}^n} \\ 2 \langle x,x_t \rangle_{\mathbb{R}^n} 
\end{pmatrix} 
\] 
is the derivative with respect to $t$, and 
$\wedge$ is defined via 
\[ a \wedge b ( c ) = \langle a , c \rangle b 
- \langle b , c \rangle a \; , \] 
identifying $\wedge^2 \mathbb{R}^{n+1,1}$ with $so_{n+1,1}$.  

\subsection{Cross ratios}

Another tool we need is the cross ratio, computed for four points in $\mathbb{R}^n$, for any $n \geq 2$.  
When real, the cross ratio we describe below 
is equivalent to the well known definition of the cross ratio in the complex plane regardless of choice of 
$n$, and being real means the four points lie on a circle in $\mathbb{R}^n$ (one reference for this is \cite{St}).  
When $n \leq 4$ one can use $2 \times 2$ matrix models or quaternionic models to compute it 
(see \cite{NORYPJ}, for example).  However,
 for general $n$ we can use the 
lightcone model as follows: for the cross ratio to be real, we require that the lifts $Y$, $Y_1$, 
$Y_{12}$, $Y_2$ of four points $y$, $y_1$, $y_{12}$, $y_2$ in $\mathbb{R}^n$ lie in a three 
dimensional space slicing the light cone (equivalent to concircularity of $y$, $y_1$, $y_{12}$, $y_2$), 
and then one can confirm that the cross ratio satisfies (see \cite{BHRS}, \cite{NORYPJ}, \cite{H}) 
\[ \text{cr}(y,y_1,y_{12},y_2) = 
\frac{\langle Y,Y_1  \rangle \langle Y_{12},Y_2  \rangle+\langle Y,Y_2  \rangle \langle Y_1,Y_{12}  \rangle - 
\langle Y,Y_{12}  \rangle \langle Y_1,Y_2  \rangle}{2 \langle Y,Y_2  \rangle \langle Y_1,Y_{12}  \rangle} \; . \]
This is one way to compute the cross ratio of concircular points for any value of $n$.  

In the case of two curves $y(t)$ and $z(t)$ with a common circle congruence, that is, a curve of circles $\frak{C}(t)$ so 
that each circle $\frak{C}(t)$ is tangent to both curves at the corresponding points $x(t)$ and $y(t)$, we can define an 
infinitesimal version of the cross ratio by taking a limit (see \cite{BHMR} for original development of the notion, and 
\cite{CRS}, \cite{Y} for examples of its application): 
\[ \lim_{\Delta t \to 0} (\Delta t)^{-2} \cdot \text{cr}(y(t),y(t+\Delta t),z(t+\Delta t),z(t)) \; . \]
This definition is suitable only when such a circle congruence exists.  In particular, existence of that 
circle congruence guarantees this limit is real-valued.  

We remark that cross ratios can be defined quite cleanly from a gauge theoretic viewpoint, with numerous  
sources for this (see works by F.E. Burstall, or Equation (3.20) in \cite{NORYPJ} for example).  

\subsection{Isothermic nets}\label{sec3}

We now consider general curvature line systems, and the subclass of isothermic nets.  Some suitable references are 
\cite{Bi}, \cite{BMS}, \cite{BS}, \cite{NORYPJ}, \cite{H}.  

A smooth ($d$-dimensional) curvature line system 
in $n$-dimensional Euclidean space $\mathbb{R}^n$ is a smooth immersion 
\[f=f(x_1,...,x_d) : \mathbb{R}^d \to \mathbb{R}^n \; , \;\;\; d \leq n \] 
such that any coordinate curve given by fixing all but one $x_i$ is a 
curvature line of any larger coordinate submanifold in the image 
of $f$ (where some number of coordinates are fixed).  Note that 
to have well-defined curvature lines in codimensions greater than one 
requires flatness of the respective normal bundles.  
In particular, the coordinate curves with distinct principal curvatures 
will all be orthogonal to 
each other at points where they intersect.  
In fact, we assume such orthogonality at all points 
regardless of the values of the principal curvatures.  
We can partially or fully 
discretize this by having some or all coordinate variables $x_i$ lie 
in the integers $\mathbb{Z}$ rather than $\mathbb{R}$.  Then, on $2$-dimensional 
coordinate subnets $\mathcal{S}$, we can define the notion of coordinate curvature lines as follows \cite{Bo}, \cite{BMS} 
(see Figure \ref{fig:diagrams}):
\begin{itemize}
\item when $\mathcal{S}$ is smooth, its variable coordinates are curvature lines in the usual smooth sense, 
\item when $\mathcal{S}$ is fully discrete, its elementary 
quadrilaterals are concircular, that is, their four vertices lie on a circle (for an analogous notion to 
immersedness, it suffices in our setting to assume the quadrilaterals each have four distinct vertices), 
and 
\item when $\mathcal{S}$ is semi-discrete, the 
corresponding points on its adjacent smooth curves have common 
tangential circles (this can be called a Ribaucour transform from one curve to the other, and for an analogous 
notion to immersedness here, it would suffice to assume the two corresponding points 
do not coincide).  
\end{itemize}

\begin{remark} 
Amongst the discrete directions, an alternate 
description using notions from M\"obius geometry can 
be found in \cite{HS}, using orthogonality properties and intersection properties of 
families of contact elements defined at the vertices of the system. 
\end{remark}

\begin{figure}
    \includegraphics[width=30mm]{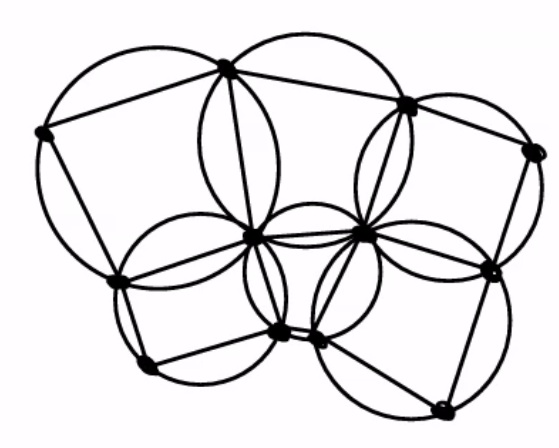}
    \includegraphics[width=30mm]{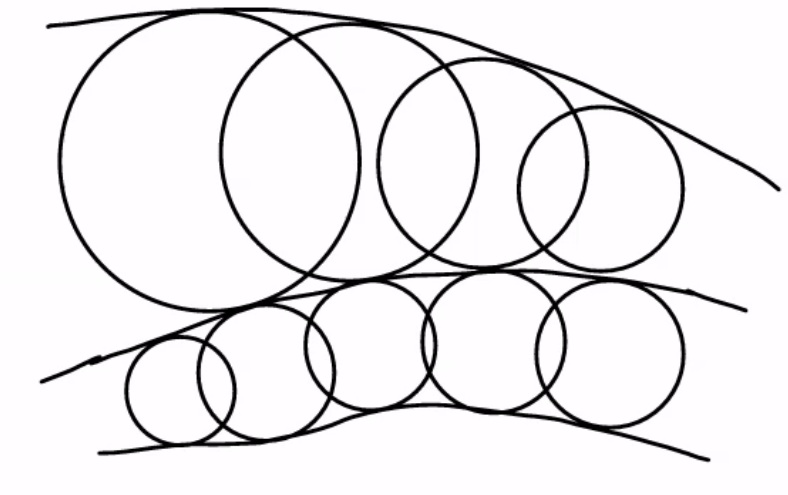}
	\caption{Fully and semi discrete $2$-subnets of orthogonal nets.}
	\label{fig:diagrams}
\end{figure}

If additionally, for any $2$-subnet $\mathcal{S}$ determined by two coordinates $(x_i,x_j)$, 
the following conditions hold 
\begin{enumerate}
\item if 
those two coordinates are both smooth, then they are conformal as well as curvature lines, 
\item if those two coordinates are both discrete, then they are 
discrete isothermic, that is, the cross ratios of the quadrilaterals factorize: 
for the cross ratios $\text{cr}_1$, $\text{cr}_2$ $\text{cr}_3$, $\text{cr}_4$ of the four quadrilaterals ordered counterclockwise 
about any chosen vertex, we have 
\[ (\text{cr}_1)(\text{cr}_2)^{-1}(\text{cr}_3)(\text{cr}_4)^{-1} = 1 \; , \]
and it can be immediately shown this is 
equivalent to the cross ratio as a function on 
faces of $\mathcal{S}$ (faces with lower left vertex having $(x_i,x_j)$ as coordinates) taking the form 
\[ \text{cr} = \frac{a(x_i)}{b(x_j)} \; , 
\]
where the function $a(x_i)$, resp. $b(x_j)$, on the edges where $x_i$, resp. $x_j$, varies 
depends only on $x_i$, resp. $x_j$.  
\item if $x_i$ is smooth and $x_j$ is discrete, then $\mathcal{S}$ is semi-discrete isothermic, that is, the 
infinitesimal cross ratios allowed by the internal circle congruence factorize, i.e. $\text{cr} = a(x_i)/b(x_j)$ 
holds, where $a(x_i)$ is now a function of a smooth variable, 
\end{enumerate}
then we call the orthogonal net an {\em isothermic} net.  

\subsection{Darboux transforms of curves via M\"obius geometry}\label{sec4}

In the case of a smooth isothermic surface in $\mathbb{R}^3$, a Darboux transform is another isothermic surface so that 
both surfaces have a common sphere congruence, and for which the correspondence of points between the two surfaces given by the 
sphere congruence preserves conformal structure and curvature lines.  One approach for computing such 
Darboux transformations is to lift the situation up to M\"obius geometry and frame the problem in terms of 
parallel sections of certain flat connections that isothermic surfaces always have.  (See \cite{Bnext}, 
\cite{H}, \cite{NORYPJ}, and we describe below this 
kind of lifting to M\"obius geometry for the case of curves.)  There is also a strictly more general class of transformations, 
namely Ribaucour transformations, where 
the pointwise correspondence preserves curvature lines but not necessarily conformal structure.  

We can similarly consider transformations of curves in the plane $\mathbb{R}^2$ or higher dimensional Euclidean spaces $\mathbb{R}^n$, 
where now there is a common circle congruence for both the original and transformed curves, which 
determines the pointwise correspondence between the two curves.  Preserving notions like principal curvature lines and conformal 
structure is now a mute point, but we can then insert a polarization on the common domain of the two curve maps that encodes the 
infinitesimal cross ratios allowed by the circle congruence, as seen in \cite{BHMR} and which we describe in the remark at the 
end of this subsection.  

Without inserting a polarization, we have a Ribaucour transformation of a curve, and once an appropriate polarization is given, we 
call it a Darboux transformation.  The polarization can be inserted after the transform is given, so Ribaucour transforms are essentially 
equivalent to Darboux transforms in the case of a curve.  However, when we wish to consider repeated Darboux transformations 
all preserving one choice of polarization, this becomes a strictly smaller class of objects than the collection of sequences of Ribaucour transformations. 
In the proof of Theorem \ref{mainthm}, we in fact do need repeated Darboux transforms with common polarization to produce $k$-tori 
that are isothermic.  

For computing Darboux transforms, we can employ M\"obius geometry.  We suggest other sources 
for proofs of statements, primarily \cite{BHMR}, and \cite{CRS} as well.  One could consider the transforms in the 
context of Euclidean geometry using a cross ratio condition.  
However, we employ an approach below that uses the lightcone model and connections, for two reasons: 
\begin{itemize}
\item to keep the analogy 
with how Darboux transforms of surfaces are naturally considered in M\"obius geometry, and 
\item to formulate 
closing conditions on the transforms in the next subsection 
in terms of existence of lightlike eigenvectors of certain Lorentz orthogonal maps.  
\end{itemize}

It can be shown that a 
Darboux transformation $\hat x(t)$ of $x(t)$, considered in $\mathbb{R}^n$, is a curve 
such that the corresponding lift $\hat X(t)$ satisfies 
\[ \nabla^\mu (r(t) \hat{X}(t)) = 0 \]
for some particular value of $\lambda=\mu$ and some choice of scalar function $r(t)$, that is, some 
scalar multiple of $\hat X(t)$ is a parallel section of $\nabla^\mu$.  We call $\mu$ the spectral parameter of the transformation.  

This provides a computational method for finding Darboux transformations.  

\begin{remark}
We can also describe Darboux transformations in terms of infinitesimal cross ratios allowed by the common circle 
congruence, and corresponding points of $x(t)$ and $\hat x(t)$ will then have infinitesimal 
cross ratios equaling $\mu/m$.  See \cite{BHMR}, \cite{CRS}, \cite{Y}.  
\end{remark}

\subsection{Closing Darboux transforms of loops}
Now we consider a smooth closed curve in $\mathbb{R}^n$, that is, a periodic loop 
\[ x=x(t): \mathbb{R} \to \mathbb{R}^n \; , \;\;\; x(t+2\pi)=x(t) \;\; \text{for all} \; t \; , \]
with polarization $\frac{dt^2}{m}$ that is also $2 \pi$-periodic.  

It is known that closed Darboux transforms (i.e. loops) exist for infinitely many 
nonzero values of $\mu$ for large classes of curves, see \cite{CLO5} (and also \cite{CLO}, \cite{CLO2}, \cite{CLO3}, \cite{CLO4} for the case of smooth and discrete circles).  
In particular, this is the case for every nonzero value of $\mu$ when $x(t)$ is a circle with arc-length 
polarization so that $m=1$ when the circle is arc-length parametrized.  
We can consider this via the parallel transport map, call it $\textbf{P}$, of $\nabla^\mu$ along $t \in [0,2 \pi]$, 
which is an isometry of $\mathbb{R}^{n+1,1}$ since $\nabla^\mu$ is an $SO_{n+1,1}$-connection.  
Taking lightlike eigenvectors of this map as initial conditons will result in closed Darboux 
transformations.  

\begin{example}\label{exa}
In fact, giving $\mathbb{R}^n$ the standard orthonormal basis $e_j$, and taking the initial curve as the circle 
\[ x(t) = (\cos t,\sin t, 0, ... , 0) \; , \] 
it was shown in \cite{CLO2} that for each nonzero $\mu$ there is a closed Darboux transform $\hat x (t)$ lying 
in the same plane that contains the circle $x(t)$, that is, within $\text{span}\{e_1,e_2\}$, and not lying within the 
circle itself.  These Darboux transformations can be considered in the plane 
by the methods in the works \cite{CLO}, \cite{CLO2}, \cite{CLO3}, \cite{CLO4}, or equivalently, by 
lifting to $\mathbb{R}^{n+1,1}$ and applying the approach described above, and then projecting back to $\text{span}\{e_1,e_2\}$.  
Applying a M\"obius transformation that fixes the circle $x(t)$, we can move $\hat x (t)$ into 
$\text{span}\{e_1,e_2,e_j\}$ for some $j$ between $3$ and $n$, but not contained in the plane $\text{span}\{e_1,e_2\}$ itself.  
In this way, we can create $n-2$ distinct closed Darboux transforms $x_j(t)$ of $x(t)$ with distinct spectral parameters 
$\mu_j$, for $j=3,...,n$ ($\hat x(t)$ now renamed to $x_j(t)$), which lie in $\text{span}\{e_1,e_2,e_j\}$ but not entirely within $\text{span}\{e_1,e_2\}$.  
We can then extend by permutability to an $(n-2)$-dimensional Bianchi cube, as in the next subsection, 
of closed Darboux transformations whose first level of transforms is these $n-2$ 
transforms $x_j(t)$, and the set 
\[ \{ x(t) \, | \, t \in [0,2 \pi)\} \cup \{  x_j(t) \, | \, t \in [0,2 \pi)\} \]
is full in $\text{span}\{e_1,e_2,e_j\}$ for each $j$.  
\end{example}

\subsection{Resulting Bianchi cubes}

To fill out the Bianchi cube of a first layer of closed Darboux transformations (some suitable references are 
\cite{Bi}, \cite{BS}, \cite{Bnext2}, \cite{H}, \cite{NORYPJ}), 
we can proceed as follows: 
Given a  closed curve $y(t)$ and two closed Darboux transforms $y_1(t)$ and $y_2(t)$ with distinct 
spectral parameter values $\nu_1$ and $\nu_2$, respectively, a common Darboux transform 
loop $y_{12}(t)$ is given algebraically by 
\[ \text{cr}(y(t),y_1(t),y_{12}(t),y_2(t))=\frac{\nu_1}{\nu_2} \; , \] 
where $y_{12}(t)$ has spectral parameter $\nu_2$, resp. $\nu_1$, with respect to $y_1(t)$, resp. $y_2(t)$.  
This is called {\em permutability}.  
This can be repeated at all locations within the Bianchi cube to produce the complete cube, and from this algebraic construction it is 
clear that if the first level of Darboux transforms are all closed, then all curves in the complete cube are closed as well.  
Furthermore, this cube is a discretized isothermic net with $d=n-1$, and any subnet is also discretized isothermic.  

\begin{remark}
In fact, one can circumvent the use of cross ratios by describing permutability of Darboux transformations in terms of gauge relations between 
the connections described in the previous section, see \cite{BCHPR}, \cite{NORYPJ} for more on this.  
\end{remark}

\subsection{Darboux transforms of discrete curves}
A similar description can be made when the initial curve $x$ is a discrete loop, and we refer the 
reader to \cite{CLO2}, \cite{Nagoya}, \cite{NORYPJ}, \cite{BCHPR}, \cite{CRS2}, \cite{BHRS}  
for information on this.  Here we simply note what the corresponding discrete 
objects are:
\begin{itemize}
\item the discrete loop is now of the form 
\[
x_j : \{0,1,2,...,\ell \} \to \mathbb{R}^n \; , \;\;\; x_\ell=x_0 
\] for some natural number $\ell \geq 3$, 
with the corresponding discrete lift $X_j$ defined the same way as in the smooth case.  
\item The polarization is now $1/m_{j,j+1}$ with $m_{j.j+1}$ a function defined on edges 
$j,j+1$.  
\item the discrete flat connections are (with $m=m_{j,j+1}$) 
\[ \hspace{0.4in} 
\nabla_{j+1,j}^\lambda Y = Y+\tfrac{\lambda}{m \langle X_j,X_{j+1} \rangle} 
\left( \tfrac{m}{m-\lambda} \langle Y,X_j \rangle X_{j+1}- \langle Y,X_{j+1} 
\rangle X_j \right) , 
\] mapping $T_{X_j} \mathbb{R}^{n+1,1}$ to $T_{X_{j+1}} \mathbb{R}^{n+1,1}$.  
Then the parallel section condition for defining a 
Darboux transform remains unchanged, but $r_j$ is now a discrete 
function.  However, this type of formulation is more useful for the case of discrete surfaces, and 
use of such discrete connections and discrete parallel sections for discrete curves 
can be avoided if one takes the definition of Darboux transforms to be as in the next item. 
\item The curve $x_j$ and its Darboux transform $\hat x_j$ with spectral parameter 
$\mu$ satisfy that $x_j$, $x_{j+1}$, $\hat x_{j+1}$ and $\hat x_j$ lie on a circle and 
have cross ratio $\mu/m_{j,j+1}$.  
\end{itemize}

\begin{figure}
    \includegraphics[width=50mm]{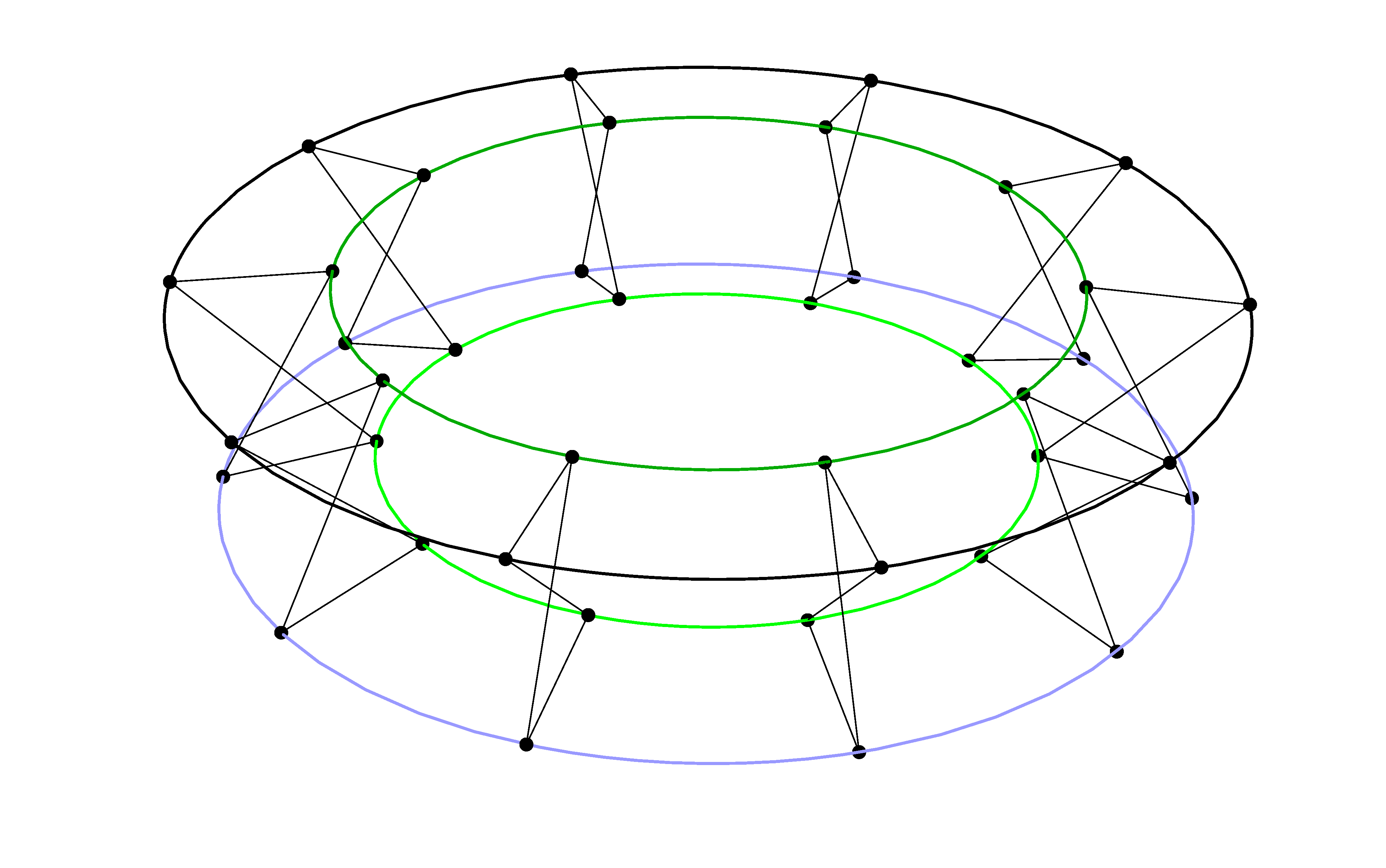}
    \includegraphics[width=50mm]{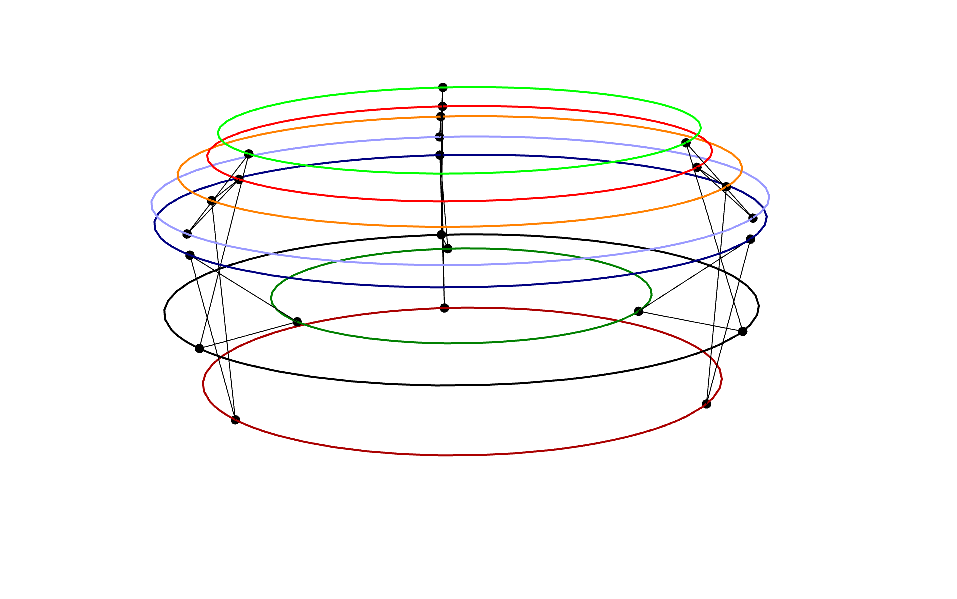}
	\caption{Semi-discrete isothermic $2$-tori with coarse meshes consisting of circles 
                     in the smooth direction in $\mathbb{R}^3$, with polarizing function $m=1$. 
                     On the left, the thickest-drawn curve is the original circle $x$, the two adjacent
                     circles are Darboux transforms with spectral parameter $\mu=3$ 
                     for the upper one and $\mu=8$ for the lower one, 
                     and the final bottom circle completes the Bianchi cube via permutability.  
                     A more complicated example still with smooth circles is 
                     shown on the right.}
	\label{fig:SD-tori-1}
\end{figure}

\begin{figure}
    \includegraphics[width=55mm]{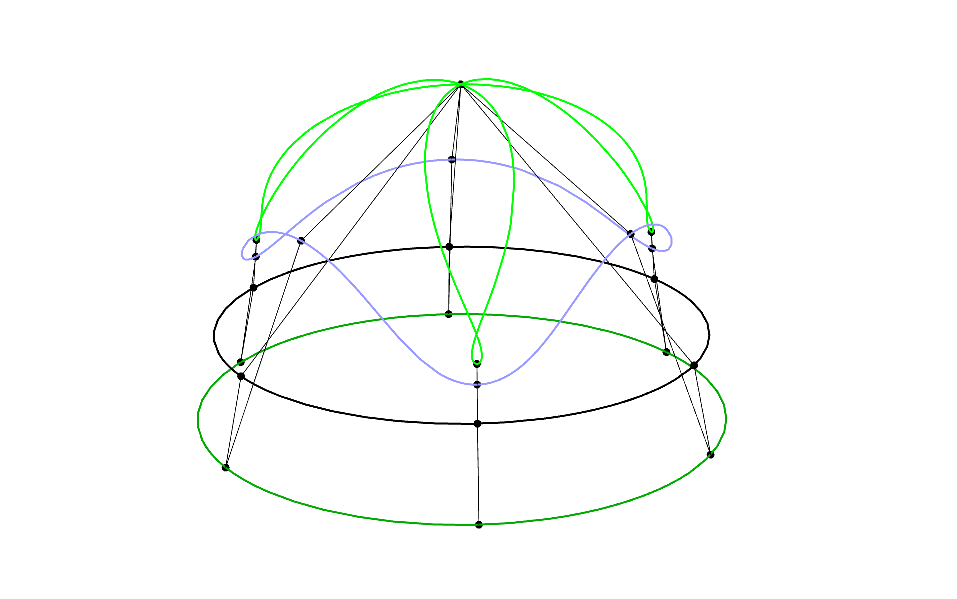}
    \includegraphics[width=45mm]{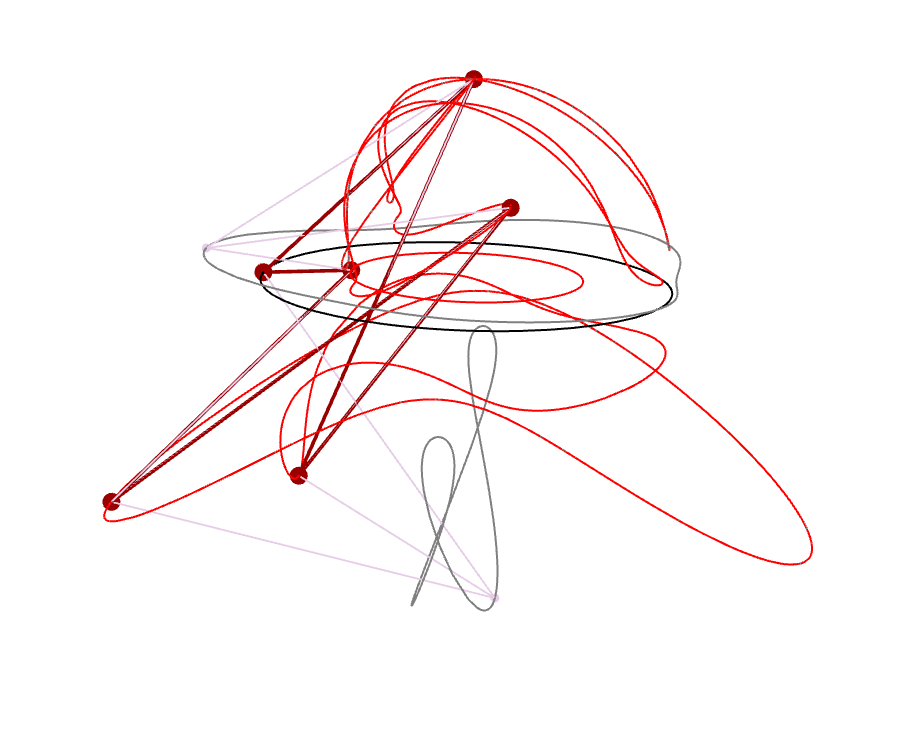}
	\caption{Semi-discrete isothermic $2$-tori with coarse meshes in $\mathbb{R}^3$ 
                        whose smooth curves are not all circular and are obtained by Darboux 
                      transforms at resonance points.}
\label{fig33}
\end{figure}

\begin{figure}
    \includegraphics[width=37mm]{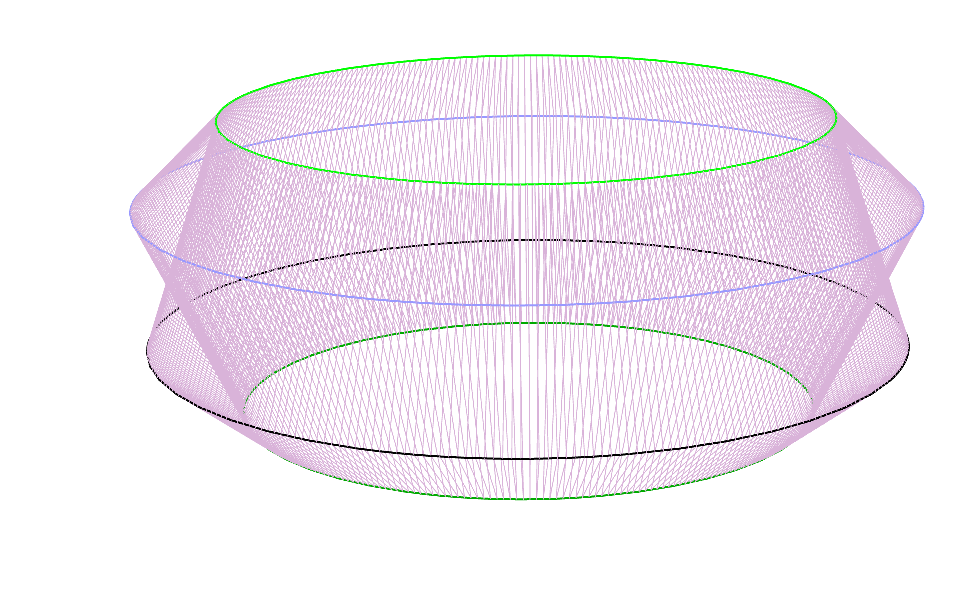}
    \includegraphics[width=38mm]{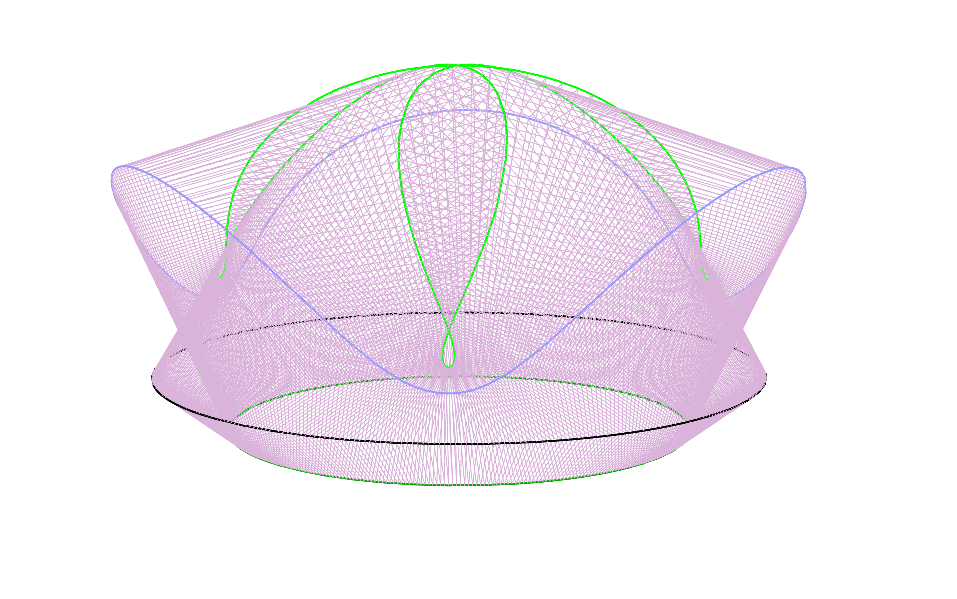}
\hspace{-0.28in}
    \includegraphics[width=38mm]{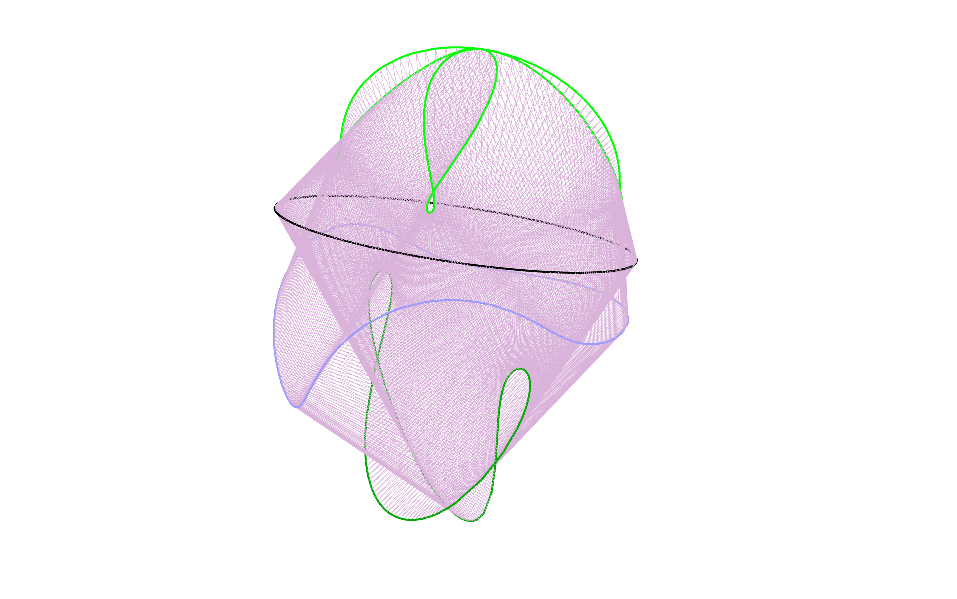}
\vspace{-0.3in}
	\caption{Three semi-discrete tori in $\mathbb{R}^3$ with finer meshes.  Each uses a circle 
as inital curve, but the transformed curves are non-circular in the two right figures, using 
resonance points.}
\label{fig44}
\end{figure}

\begin{figure}
    \includegraphics[width=35mm]{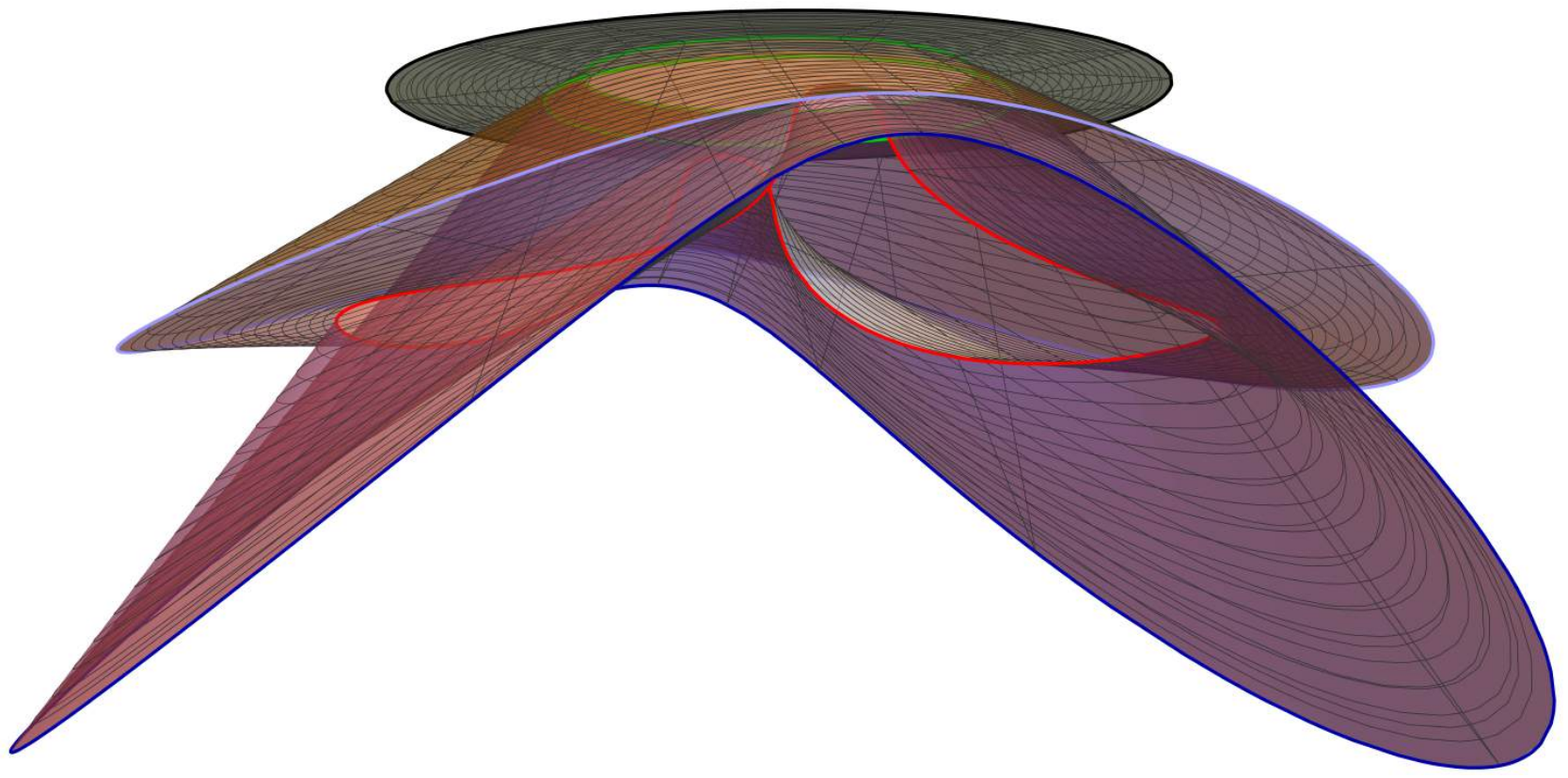}
    \includegraphics[width=33mm]{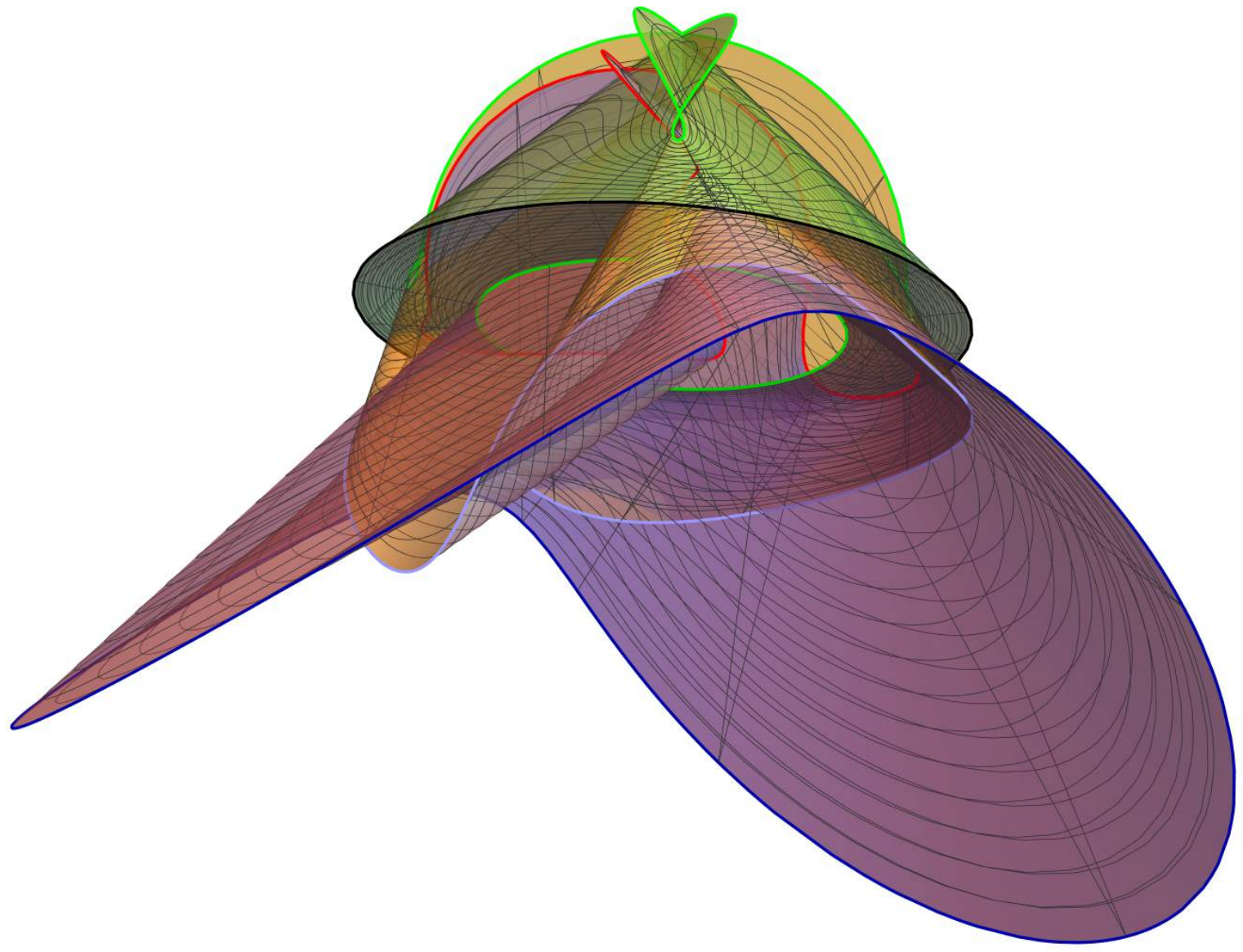}
    \includegraphics[width=33mm]{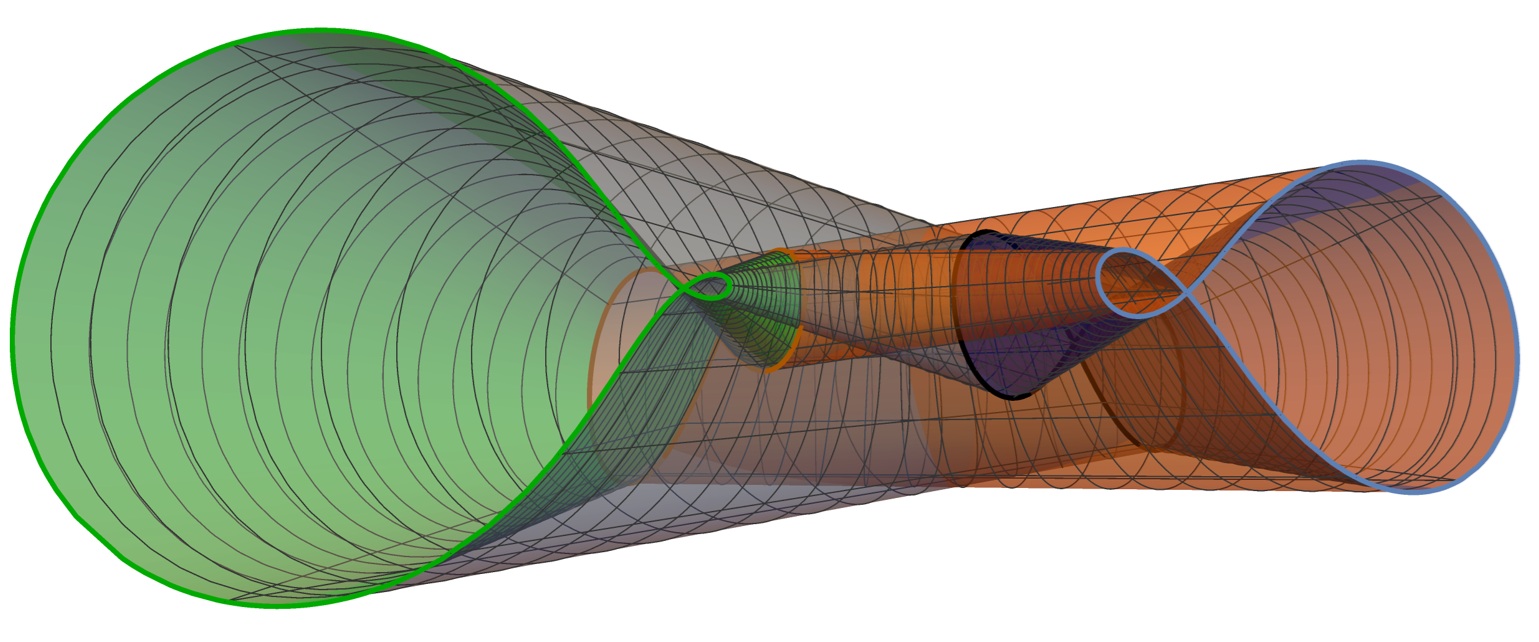}
	\caption{Three semi-discrete tori in $\mathbb{R}^3$ with finer meshes.  
The left two figures use a circle 
as inital curve, with non-circular transformed curves via resonance points.  The right figure 
starts from a figure-eight elastic curve and uses non-resonant spectral parameters to obtain 
non-trivial Darboux transforms.}
\label{fig555}
\end{figure}

\begin{figure}\label{fig66}
    \includegraphics[width=66mm]{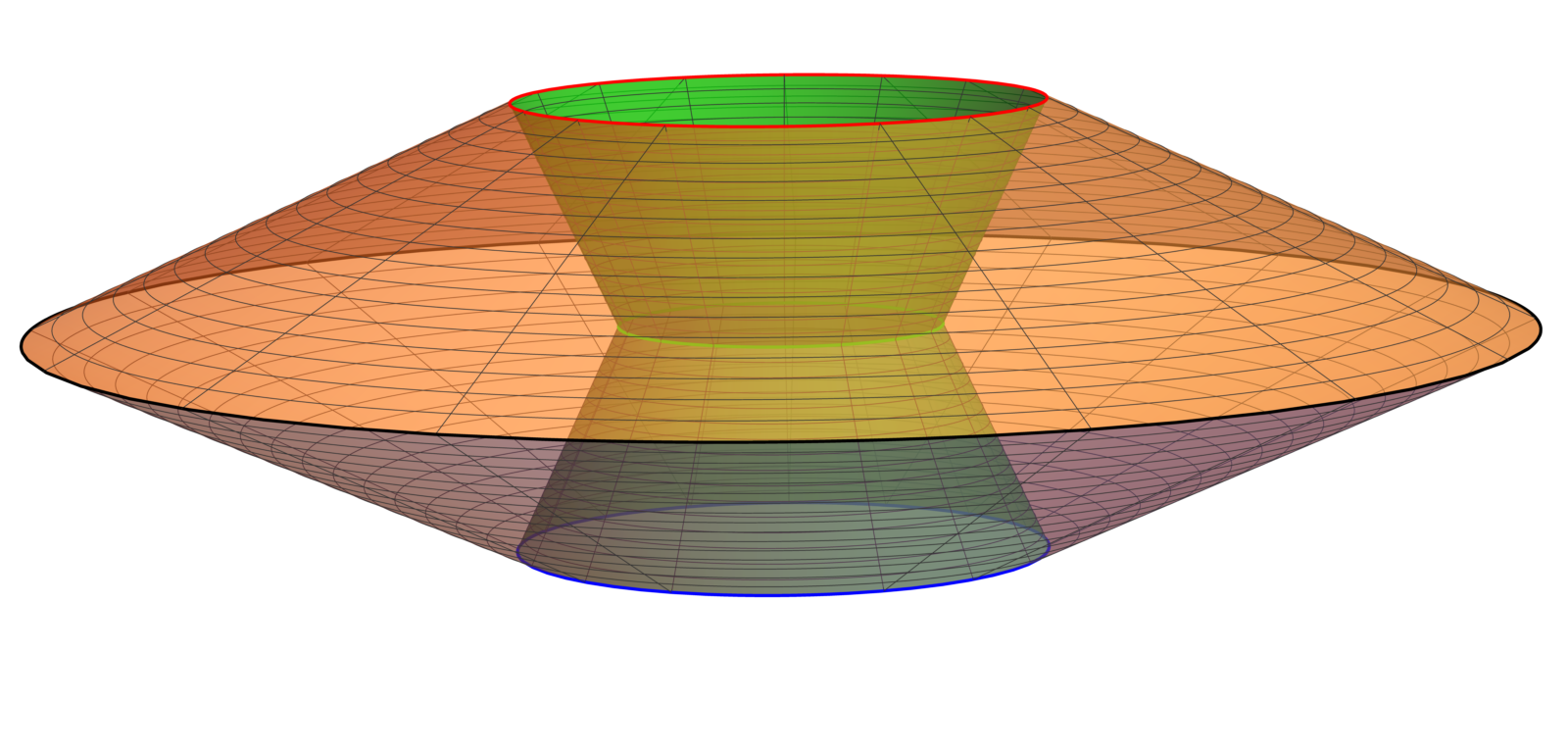}
    \includegraphics[width=40mm]{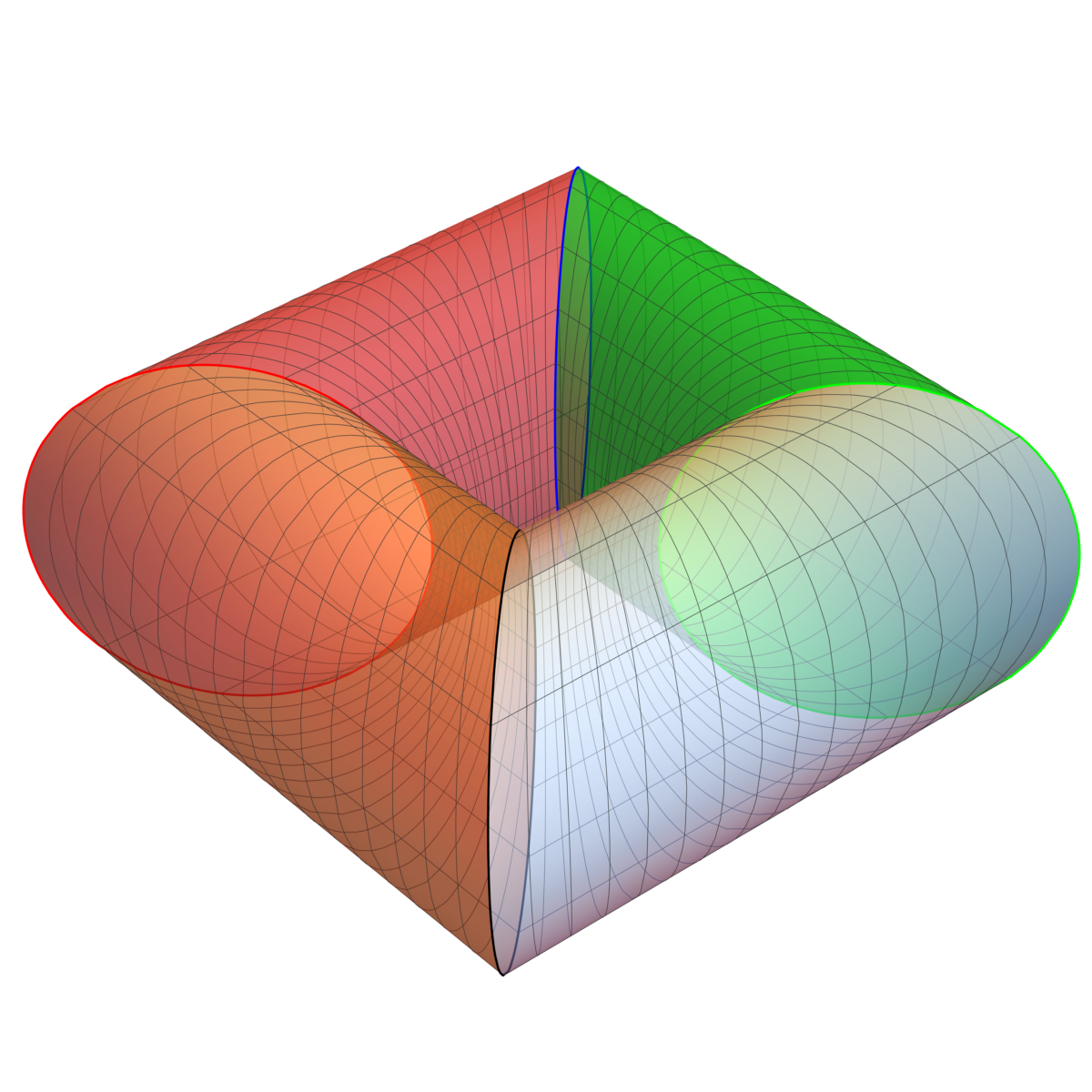}
	\caption{Embedded semi-discrete examples created by imposing extra symmetries on 
the Bianchi cube, which allows for embedded discrete circular curvature lines.}
\end{figure}

\section{Proof of Theorem \ref{mainthm}}\label{sec2}

\begin{proof}
Consider the complete Bianchi cube made from closed Darboux transforms $x_j(t)$ 
of a circle $x(t)$ -- like in Example \ref{exa} -- as the first layer of transforms.  Stepping through particular curves, with each selected curve 
adjacent to the previous one, we can create a ``discrete curve of curves" starting 
and ending at $x(t)$ and including all $x_j(t)$ for $j=3,...,n$.  This gives us a semidiscrete $2$-dimensional torus that is full in 
$\mathbb{R}^n$, and is isothermic (that is, it is a discretized isothermic net).  More generally, taking a ``discrete 
$(k-1)$-dimensional surface of curves" within the Bianchi cube, we can build a semidiscrete isothermic 
$k$-dimensional torus, again full in $\mathbb{R}^n$.  Similarly, one could replace $x(t)$ by a discrete circle 
and follow the same approach to create a full discrete isothermic $k$-dimensional torus once again full in $\mathbb{R}^n$.  
\end{proof}

\begin{remark}
When the linear parallel transport map $\textbf{P}$ becomes $\pm \text{id}$, for 
which we call the associated 
$\mu$ a {\em resonance point}, then any lightlike initial condition will result in a 
closed Darboux transform (again see \cite{CLO}, \cite{CLO2}, \cite{CLO3}, \cite{CLO4}).  
Many planar loops have resonance points (see \cite{CLO5}), and circles in particular have infinitely many resonance points.  Thus, when the initial curve $x(t)$ is a circle, 
we could choose all the $\mu_j$ to be such points.  While this is not 
necessary for the argument above proving Theorem \ref{mainthm}, this has the advantage of producing more interesting semidiscrete tori, 
as most values of $\mu$ will give only circles for $\hat x (t)$ (see Figure \ref{fig:SD-tori-1} and the lefthand side of Figure 
\ref{fig44}), but resonance points will provide noncircular 
$\hat x (t)$ (see Figure \ref{fig33}, the two righthand graphics in Figure \ref{fig44} and the two lefthand graphics in 
Figure \ref{fig555}).  

Also, should we have a noncircular closed initial curve $x(t)$ with infinitely many resonance points, 
we could use $n-2$ of those resonance points as our values for $\mu_j$.  We 
can then freely choose the initial 
conditions for the curves $x_j(t)$ and be guaranteed those curves also close, so that the entire Bianchi cube has closing 
curves.  If those initial conditions, together with $x(t)$, form a full set in $\mathbb{R}^n$, the above argument, with only slight 
possible adjustments for the dimensions of subspaces containing particular curves, will still apply.

Generally speaking, resonance points give the most interesting closed Darboux transforms, but interesting 
examples can occur for non-resonant points as well.  
\end{remark}

\begin{remark}
We can consider the above argument in a more general setting, again where the initial closed curve $x(t)$ may not be a circle.  
We could modify the argument only slightly to allow $x(t)$ and closed $x_j(t)$ to be in subspaces of other dimensions.  For 
example, we could consider a noncircular $x(t)$ full in $\text{span}\{e_1,e_2,e_3\}$ and $x_j(t)$ hopefully full in 
$\text{span}\{e_1,e_2,e_3,e_j\}$ for each $j \geq 4$, 
so in effect we are taking $n=4$ when producing each individual $x_j(t)$.  
From \cite{GQ}, it follows that $\textbf{P}$ has real lightlike eigenvectors for 
$A \in SO_{n+1,1}$ if $n$ is even.  Should one of those eigenvectors lie outside the subspace containing $x(t)$, we could then provide 
examples that prove Theorem \ref{mainthm} where no single curve in the torus is a circle.  

Because we are not guaranteed that the lightlike eigenvectors are situated so as to produce $x_j(t)$ in higher dimensional 
subspaces than that of $x(t)$, we cannot yet use this more general argument to prove the theorem.  However, numerical 
results show that there are curves for which eigenvectors are situated so that in fact this stumbling block does not occur, see 
the right-hand graphic in Figure \ref{fig555}.
\end{remark}

\noindent {\bf Acknowledgements.} The authors thank Fran Burstall, Joseph Cho, Yuta Ogata and the referee for 
helpful discussions and suggestions.

\end{document}